\documentclass[]{amsart}

\usepackage{amssymb, mathrsfs, color}

\usepackage{tikz}

\newtheorem{thm}{Theorem}[section]
\newtheorem{cor}[thm]{Corollary}
\newtheorem{lem}[thm]{Lemma}

\newtheorem{prop}[thm]{Proposition}
\newtheorem{prob}[thm]{Problem}

\newtheorem{rem}[thm]{Remark}

\numberwithin{equation}{section}

\renewcommand{\phi}{\varphi}

\newcommand{\R}{\mathbb{R}}

\newcommand{\D}{\mathbb{D}}

\newcommand{\eps}{\varepsilon}

\newcommand{\re}{\mbox{\rm Re\,}}
\newcommand{\im}{\mbox{\rm Im\,}}
\newcommand{\diam}{\mbox{\rm diam\,}}

\begin{document}

\title[Some Integral operators acting on $H^\infty$]
{Some integral operators acting on $H^\infty$}

\author[A. Anderson]{Austin Anderson}
\address{Department of Mathematics, University of Hawaii, Honolulu, Hawaii 96822}
\email{austina@hawaii.edu}
\author[M. Jovovic]{Mirjana Jovovic}
\address{Department of Mathematics, University of Hawaii, Honolulu, Hawaii 96822}
\email{jovovic@math.hawaii.edu}
\author[W. Smith]{Wayne Smith}
\address{Department of Mathematics, University of Hawaii, Honolulu, Hawaii 96822}
\email{wayne@math.hawaii.edu}

\thanks {}
\subjclass [2000] {}

\begin{abstract}
Let $f$ and $g$ be analytic on the unit disc $\D$.  The  integral operator $T_g$ is defined by
$ T_g f(z) = \int_0^z f(t)g'(t)\,dt$,  $z\in \D$.  The problem considered is characterizing
those symbols $g$ for which $T_g$ acting on $H^\infty$,
the space of bounded analytic functions on $\D$, is bounded or compact.
When the symbol is univalent, these become questions in univalent function theory.
The corresponding problems for the companion operator, $ S_g f(z)= \int_0^z f'(t)g(t)\, dt$,
acting on $H^\infty$ are also studied.
\end{abstract}
\maketitle

\section{Introduction}

Let $\D$ denote the unit disk $\{z : |z| < 1\}$ and $H(\D)$ the set of analytic functions on $\D$.
The operator $T_g$ with symbol $g \in H(\D)$, defined by
  $$T_g f(z) = \int_0^z f(t)g'(t)\,dt, \qquad z \in \D, f \in H(\D),$$
has attracted interest as a generalized C\'esaro or Volterra operator.
For the multiplication operator
  $$M_gf(z) = f(z)g(z),$$
and the companion operator
   $$S_gf(z) = \int_0^z f'(t)g(t)\, dt, $$
integration by parts gives that
\begin{equation}\label{prod}
  M_gf(z) = f(0)g(0) + T_gf(z) + S_gf(z).
\end{equation}
For a Banach space $X \subseteq H(\D)$ and a linear operator $L_g$ on $X$, let
  $$L[X] = \{g \in H(\D): L_g \text{ is bounded on } X\},$$
and
  $$L_o[X] = \{g \in H(\D): L_g \text{ is compact on } X\}.$$

It is easily checked from the definitions that the operators $T_g$, $S_g$, and $M_g$
are all linear in the parameter $g$.  Also $M_{g_1g_2}=M_{g_1}M_{g_2}$, and
$S_{g_1g_2}=S_{g_1}S_{g_2}$.  Hence $M[X]$ and $S[X]$ are always subalgebras of $H(\D)$.
This is not the case for $T[X]$, in general.

Christian Pommerenke \cite{Pom1} first noted that $T[H^2] = \rm{BMOA}$ by the Littlewood-Paley identity for the Hardy space $H^2$, and he used this fact to give a proof of the analytic John-Nirenberg inequality.  $T[X]$ has been determined for a variety of spaces $X$, including $X = H^p$, $1 \leq p < \infty$, by Aleman and Siskakis \cite{AS}, and $X = \rm{BMOA}$ by Siskakis and Zhao \cite{SZ}.  The condition characterizing compactness of an operator is typically a little-oh version of the condition characterizing boundedness, as
for example $T_o[H^2] = \rm{VMOA}$ \cite{AS}.  The companion operator $S_g$ is easier to analyze in many situations; see Proposition \ref{S_gbdd}.

  In this paper we begin the study of $T[H^{\infty}]$ and $T_o[H^{\infty}]$, where $H^\infty$ is the
  usual Banach space of bounded holomorphic functions on $\D$ with supremmum norm.
It turns out (Proposition \ref{subinfty}) that
$T[H^{\infty}] = S[H^{\infty}]$, and hence  $T[H^{\infty}]$ is an algebra.  Moreover,
$T[H^{\infty}]\subseteq H^{\infty}$, yet this containment is proper.  Examples of bounded functions outside of $T[H^{\infty}]$ include some Blaschke products (Theorem \ref{BReal}) and some functions in the disk algebra $A$ (Proposition \ref{deathspiral}).
An obvious sufficient condition for $T_g$ to be bounded on $H^{\infty}$ is that $g$ has
bounded radial variation, but whether this is necessary is an open question.  When the symbol $g$ is univalent, a change of variables shows membership of $g$ in $T[H^{\infty}]$ becomes a question in univalent function theory; see Proposition \ref{univalent}.  Regarding compactness, we note that $S_o[H^\infty] = \{0\}$ in Proposition \ref{S_o}, but $T_o[H^{\infty}]$ is not trivial.  We show that a function with derivative in $H^1$ induces a compact operator (Proposition \ref{cptsuff})
and that $T_o[H^{\infty}]\subseteq A$ (Theorem \ref{TosubsetA}).
The space $A$ itself provides another interesting setting for $T_g$. We show that
$T[A]=T[H^\infty]\cap A$ (Proposition \ref{TAsubTH}), and
give partial results toward characterizing $T_o[A]$.  We end the paper with a section of problems and
questions suggested by our work.

\section{$T[H^{\infty}]$}

First we examine necessary conditions for boundedness of $T_g$ and $S_g$.
Lemma \ref{nec} is analogous to a result for $M_g$; see \cite[Lemma 11]{DRS}.

\begin{lem}
\label{nec}
Let $X$ and $Y$ be Banach spaces of analytic functions, $z \in \D$, and let $\lambda_z$
and $\lambda_z'$ be linear functionals
defined by $\lambda_zf = f(z)$ and $\lambda'_zf = f'(z)$ for $f \in X \cup Y$.  Suppose
$\lambda_z$ and $\lambda'_z$ are bounded on $X$ and $Y$.

\emph{(i)} If $S_g$ maps $X$ boundedly into $Y$, then
$$|g(z)| \leq \, \|  S_g  \| \,
 \frac{\|  \lambda'_z  \|_Y }{\|  \lambda'_z  \|_X}.$$

\emph{(ii)} If $T_g$ maps $X$ boundedly into $Y$, then
$$|g'(z)| \leq \, \|  T_g  \| \,
 \frac{\|  \lambda'_z  \|_Y }{\|  \lambda_z  \|_X}.$$
 \end{lem}
  \begin{proof}
Note that, for $f \in X$,
\begin{equation}
\label{sglambda}
|f'(z)||g(z)|  = |\lambda'_zS_g(f) | \leq \, \|  \lambda'_z  \|_{ _Y} \,  \|  S_g  \| \|  f  \|_{_X}.
\end{equation}
Since $$\sup_{\|  f  \|_X = 1} |f'(z)| = \,  \|  \lambda'_z  \|_X,$$ taking the supremum of both sides
of (\ref{sglambda}) over  $\{f \in X: \|  f  \|_{_X}  = 1\}$
gives us
  $$\|  \lambda'_z  \|_{_X} |g(z)| \leq \, \|  S_g  \| \|  \lambda'_z  \|_{_Y}.$$
Hence (i) holds.
Similarly,
  $$|f(z)||g'(z)|  = |\lambda'_zT_g(f) | \leq \, \|  \lambda'_z  \|_{_Y} \,  \|  T_g  \| \|  f  \|_{_X}.$$
Taking the supremum over $\{f \in X: \|  f  \|_{_X}  = 1\}$ , we get
  $$\|  \lambda_z  \|_{ _X} |g'(z)| \leq \|  T_g  \| \|  \lambda'_z  \|_{_Y}.$$
\end{proof}

Lemma \ref{nec} explains why, in many cases, $S_g$ is easy to analyze.

\begin{prop} \label{S_gbdd}
Let $X$ be a Banach space of analytic functions on which point evaluation of the derivative is a bounded linear functional. Then $S[X] \subseteq H^{\infty}$.
\end{prop}
\begin{proof}
The proof is immediate from Lemma \ref{nec} (i) with $X = Y$.
\end{proof}

From \eqref{prod}, we see that when any two of $T_g$, $S_g$, and $M_g$ are bounded,
then so is the third.  However, there are a variety of ways in which the sets $T[X], S[X]$ and $M[X]$ might relate to each other.  We have noted that $\rm{BMOA} = T[H^2] \supsetneqq M[H^2] = S[H^2] = H^{\infty}$, and Proposition \ref{subinfty} is the first step in showing $H^{\infty} = M[H^{\infty}] \supsetneqq S[H^{\infty}] = T[H^{\infty}]$.  Other relationships are known to hold when the operators act on $\rm{BMOA}$,  eg., and other spaces; see \cite{And}.

\begin{prop} \label{subinfty}
$T[H^{\infty}] = S[H^{\infty}] \subseteq M[H^{\infty}] = H^{\infty}.$
\end{prop}
\begin{proof}
It is well known that
  $$\|\lambda'_z\|_{H^{\infty}} = \frac{1}{1 - |z|^2};$$
see for example \cite[Proposition 5.1]{Zh}.  Thus
Proposition \ref{S_gbdd} gives that $S[H^{\infty}] \subseteq H^{\infty}$.
Letting $1 \in H^{\infty}$ denote the constant function,
  $$\|g-g(0)\|_{\infty} = \|T_g 1\|_{\infty} \leq \|T_g\|_{H^{\infty}}.$$
Thus $T[H^{\infty}] \subseteq H^{\infty}$.
Combined with the fact $M[H^{\infty}] = H^{\infty}$, the result follows from (\ref{prod}).
\end{proof}

Examples will be given in Theorem \ref{BReal} and Proposition \ref{deathspiral} showing the inclusion in Proposition \ref{subinfty} is proper.  First, we shall give a sufficient condition for $g$ to be in $T[H^\infty]$.


For $0 \leq \theta < 2\pi$, we denote the radial variation of a function $g \in H(\D)$ by
  $$V(g,\theta) =  \int_0^1 |g'(te^{i\theta})|\, dt.$$
We consider the class of analytic functions on the disk with bounded radial variation, defining
  $$\rm{BRV} = \{g \in H(\D) : \sup_{\theta} V(g, \theta) < \infty \}.$$
It is clear that $\rm{BRV} \subseteq T[H^{\infty}]$, since
\begin{align*}
  \| T_g f\|_{\infty} &= \sup_{z \in \D} \left| \int_0^z f(w)g'(w) \, dw \right| \\
  &= \sup_{\theta} \left| \int_0^1 f(te^{i\theta}) g'(te^{i\theta})e^{i\theta} \, dt \right| \\
  &\leq \sup_{\theta} V(g,\theta) \|f\|_{\infty}.
\end{align*}

We formulate this observation as a proposition for later reference.

\begin{prop}\label{ubrv}
$\rm{BRV} \subseteq T[H^{\infty}]$.
\end{prop}

It is natural to ask if the inclusion in this proposition is actually an equality.
This question will be formally posed, along with an equivalent formulation for the
case when $g$ is univalent, in \S 4.
All our results are aligned with a positive answer.
In particular, see Proposition \ref{BStolz} and Theorem \ref{BReal}, which concern Blaschke products; or
see Proposition \ref{deathspiral} and Remark \ref{deathspiral2}, which concern univalent functions.

A Blaschke product $B$ with zero sequence $\{a_k\} \subseteq\D - \{0\}$ is given by
  $$B(z) =  \prod_k \frac{|a_k|}{a_k}  \frac{a_k - z}{1 - \overline a_k z}, \qquad z \in \D.$$
We call $B$ \textit{uniformly Frostman} if its zeros satisfy
\begin{equation} \label{UF}
  \sum_k \frac{1-|a_k|}{|a_k - e^{i\theta}|} < C  \text{ for some } C \text{ independent of } \theta.
\end{equation}
The terminology we give comes from a condition used by Frostman to analyze the radial limits of $B$, and Cargo \cite{Car}
used (\ref{UF}) to characterize when all subproducts of $B$ are in $\rm{BRV}$.  A corollary to
\cite[Theorem 1]{Car} is that uniformly Frostman Blaschke products are in $\rm{BRV}$.  The condition (\ref{UF}) forces the sequence $\{a_n\}$ to approach the unit circle tangentially.  In contrast, an interpolating Blaschke product with zeros in a nontangential approach region is not in $\rm{BRV}$, as we show in Proposition \ref{BStolz}.


Denote by $B_k$ the Blaschke product with the $k^\text{th}$ zero removed from the
Blaschke sequence, so that
\begin{align*}
  B_k(z)=B(z)\frac{a_k}{|a_k|}\frac{1-\overline a_k z}{a_k-z}.
\end{align*}
Also, denote by $\rho$ the pseudo-hyperbolic metric on $\D$, so
$$
\rho(z, w) = \left | \frac{z - w}{1 - \overline{w} z} \right|,\quad z,w\in\D,
$$
and recall \cite[Chap. VII]{Gar}  the following facts about interpolating sequences for $H^{\infty}$:

\begin{prop}\label{is}
The following conditions are equivalent:

(1)  $\{a_k\}$ is an interpolating sequence;

(2) There exists $\delta>0$ such that
\begin{equation}\label{interpseq}
    |B_k(a_k)|\ge \delta,\quad k\ge 1;
\end{equation}

(3) The points $a_k$ are separated, i.e. there exists $a > 0$ such that
\begin{equation*}
    \rho(a_j, a_k)\ge a,\quad j\neq k;
\end{equation*}
and ${\mu_{\{a_k\}} = \sum_k (1 - |a_k|) \delta_{a_k}}$ is a Carleson measure on the disk.

\end{prop}

Recall that $\mu$ is a Carleson measure on the disc  if there exists a constant $C$ such that  $\mu(S_{I}) \leq  C | I |$, for every Carleson square $\displaystyle{S_{I} = \{r e^{i \theta}: e^{i \theta} \in I, 1 - \frac{|I|}{2 \pi} \leq r < 1 \}}$.


\begin{lem}\label{NTL}
If $f\in$ \rm{BRV}, then $f$ has a nontangential limit at every point of the unit circle.
\end{lem}

\begin{proof}
Clearly functions in \rm{BRV}  have radial limits at all points on the unit circle, and also
are in $H^\infty$.  So the result follows from the well known fact that a function in
$H^\infty$ with a radial limit at $e^{i\theta}$ has a nontangential limit at $e^{i\theta}$;
see \cite[Exercise 14, Chapter 14]{Rud}.
\end{proof}

\begin{prop} \label{BStolz}
If $B$ is an interpolating Blaschke product with zero sequence $\{a_k\}$ contained in
a nontangential approach region, then $B\notin\,$\rm{BRV}.
\end{prop}

\begin{proof}
Assume without loss of generality that the zeros of $B$ are contained in a nontangential approach region
with vertex at 1.
Denote by $\gamma_k$ the circle with center $a_k$ and
radius $(1-|a_k|)/2$, oriented counterclockwise.  Then, by Cauchy's formula,
\begin{equation*}
    |B'(a_k)|=\frac1{2\pi}\left|\int_{\gamma_k}\frac{B(z)}{(z-a_k)^2}\, dz\right|
    \le \frac{2M_k}{(1-|a_k|)},
\end{equation*}
where $M_k=\sup\{|B(z)|\,:\, z\in \gamma_k\}$.  Since $B'(a_k)=-B_k(a_k)/(1-|a_k|^2)$,
it follows from (\ref{interpseq}) that
\begin{equation*}
    M_k\ge \frac{\delta}4.
\end{equation*}
Denote by $\Gamma$ a nontangential approach region with vertex at 1 and large enough so that
$\gamma_k\subseteq\Gamma$ for all $k$.  Then
\begin{equation*}
    \limsup_{\Gamma \backepsilon \,z\to 1} |B(z)|\ge \frac{\delta}4>0.
\end{equation*}
Since the zeros of $B$ are in $\Gamma$ and converge to 1, it follows that $B$ does not have a nontangential limit
at 1 and hence $B\notin\,$\rm{BRV} by Lemma \ref{NTL}.
\end{proof}



Next we will show that  the interpolating Blaschke products with real zeros induce unbounded integral operators on $H^{\infty}$. Let
$\displaystyle{ \Delta(a, r) = \{ z \in \D: \rho(a,z) < r \}}$  denote the
pseudo-hyperbolic disc of radius $r$ centered at $a$.  We will need the following lemma.

\begin{lem}\label{EBP}
Let  $B$ be an interpolating Blaschke product with zero sequence $\{a_k\}$.  For every $\epsilon > 0$  there exists $\gamma > 0$ such that  $|B(z)| \geq \gamma $  for all
$z \in \D\setminus\bigcup_{k} \Delta(a_k, \epsilon)$.
\end{lem}

\begin{proof}
If not, there exist $\varepsilon>0$ and a sequence
$\{b_j\}\subset\D\setminus\bigcup_{k} \Delta(a_k, \epsilon)$
such that $B(b_j)\to0$ as $j\to\infty$.  Since $B$ has no zeros in
$\D\setminus\bigcup_{k} \Delta(a_k, \epsilon)$, $|b_j|\to1$ as $j\to\infty$.
By pruning the sequence we may assume that $2(1-|b_{j+1}|)\le(1-|b_j|)$ for all $j$,
which implies that $\{b_j\}$ is an interpolating sequence; see for example \cite[Theorem 9.2]{Dur}.
Since the pseudo-hyperbolic distance from $\{a_k\}$ to $\{b_j\}$ is positive, it follows that
$\{a_k\}\cup\{b_j\}$ is an interpolating sequence; see \cite[Exercise VII.2]{Gar}.  Let
$\widetilde{B}$ be the corresponding Blaschke product, and let
$\widetilde{B_j}(z)=\widetilde{B}(z)(1-\overline{b_j}z)/(b_j-z)$.  Then
$|\widetilde{B_j}|\le |B|$ for each $j$, since the zero sequence of
$\widetilde{B_j}$ contains $\{a_k\}$.  Hence, by (\ref{interpseq}), there exists $\delta>0$
such that
$$
|B(b_j)|\ge |\widetilde{B_j}(b_j)|\ge\delta.
$$
This contradicts the assumption that $B(b_j)\to0$ and completes the proof.

\end{proof}

\begin{lem}\label{interlace}
Let  $B$ be a Blaschke product with zero sequence $\{a_k\}$ contained in $(0,1)$
and satisfying $a_k<a_{k+1}$, $k\ge 1$.  Then the
zero sequence $\{b_k\}$ of $B'$ is contained in $(0,1)$ and interlaces the sequence $\{a_k\}$;
i.e. there is exactly one zero of $B'$ in each interval $(a_k, a_{k+1})$, $k \geq 1$.
\end{lem}

\begin{proof}
By the Riemann - Hurwitz formula \cite[Theorem 5.4.1]{Bea},  each partial product $$P_n(z) = \prod_{k=1}^{n}   \frac{a_k - z}{1 - \overline a_k z}$$
of $B$ of degree $n$ has exactly $n-1$ critical points in $\D$.
Since $P_n$ is real on $(0,1)$, Rolle's Theorem shows that there is a critical point of $P_n$
in each interval $(a_k, a_{k+1})$, $1 \leq k \leq n-1$.  As this accounts for all
$n-1$ critical points,
all critical points of $P_n$ are in $(0,1)$ and interlace the sequence $\{a_k\}$,
$1 \leq k \leq n$.  Since ${P_n}' \rightarrow B'$  uniformly on compact subsets of $\D$,
it follows from Hurwitz's Theorem that the  sequence $\{b_k\}$  belongs to the interval $(0,1)$ and  interlaces the sequence  $\{a_k\}$.
\end{proof}

\begin{thm} \label{BReal}
If $B$ is an interpolating Blaschke product with zero sequence $\{a_k\}$ contained in the interval $(0,1)$, then the operator  $T_{B}$ is unbounded on $H^{\infty}$.
\end{thm}

\begin{proof}

 By Lemma \ref{interlace}, the zero sequence $\{b_k\}$   of $B'$ is contained in $(0,1)$,
and interlaces
 the sequence $\{a_k\}$.  Hence $\{b_k\}$ is a Blaschke sequence, and
we can write $B' = \tilde{B} G$, where $\tilde{B}$ is the Blaschke product with zeros $\{b_k\}$
and $G$ is never zero on $\mathbb D$.
Since $B'$ and $\tilde{B}$ are real on $(0,1)$, so is $G$.  We may assume
without loss of generality that $G(r) > 0$  for $r \in (0,1)$.

Next we show that $\{b_k\}$ satisfies condition (3) of Proposition \ref{is}, and hence is an interpolating sequence.
By \cite[Lemma 3.5]{GPV} there exist positive constants $\alpha, \beta$  such that
\begin{equation}\label{EGPV}
    |B'(z)| \geq \frac{\beta}{1 - |a_k|},
\end{equation}
for all $z \in \Delta(a_k, \alpha)$, where the pseudo-hyperbolic discs $\{\Delta(a_k, \alpha)\}_{k=1}^{\infty}$ are pairwise disjoint.
Hence $b_k \notin \Delta(a_n, \alpha)$, for all $k,n = 1,2, ...$.  Since there is exactly one $b_k$ in each interval $(a_k, a_{k+1})$,  it follows that $\rho(b_i, b_j) >  \alpha$, $i \neq j$.  It is easy to check that $\mu_{\{b_k\}}$ is a Carleson measure on $\D$. Hence condition (3) of Proposition \ref{is} is satisfied, and $\{b_k\}$ is an interpolating sequence.

 We now show that $T_B \tilde B \notin H^{\infty}$, and hence $T_B$ is unbounded on $H^{\infty}$.
For $t \in (0,1)$,   $ \tilde{B}(t) B'(t) = \tilde{B}^2 (t) G (t) \geq 0 $, and hence $ \tilde{B}(t) B'(t) =  |\tilde{B}(t)| |B'(t)|$.
If  $t \in \Delta(a_n, \frac{\alpha}{2})$, then $t \in \D\setminus\bigcup_{k} \Delta(b_k, \frac{\alpha}{2})$ and
 by Lemma \ref{EBP} there exists $\gamma > 0$ such that  $|\tilde{B}(t)| > \gamma$.
 It is well known that the euclidean diameter of $\Delta(a_k, R)$ satisfies
 $\diam(\Delta(a_k, R))\ge R(1 - |a_k|)$; see for example \cite[p. 3]{Gar}.
 Using these estimates and the estimate for $|B'|$ from (\ref{EGPV}), we have
 $$
 \int_{\Delta(a_k, \frac{\alpha}{2}) \cap (0,1)} \tilde{B}(t)B'(t) dt
 \geq \frac{\beta \gamma}{1 - |a_k|}\, \diam\left( \Delta \left( a_k,  \frac{\alpha}{2} \right) \right) \geq \frac{\alpha\beta\gamma}2 >0\,.
 $$

Therefore
$$
\lim_{r \rightarrow 1}  T_B\tilde{B} (r)
  = \lim_{r \rightarrow 1} \int_{0}^{r} \tilde{B}(t) B'(t) dt
  \geq \sum_{k=1}^{\infty} \int_{\Delta(a_k, \frac{\alpha}{2}) \cap (0,1)} \tilde{B}(t)B'(t) dt
  =\infty,
$$
and so $T_B \tilde B \notin H^{\infty}$.  This completes the proof.

\end{proof}


Since $T[H^{\infty}]$ is an algebra properly contained in $H^{\infty}$, we are brought to consider the disc algebra $A$  of analytic functions on $\D$ which extend to be continuous on $\overline{\D}$.
As noted earlier, uniformly Frostman Blaschke products are in $T[H^{\infty}]$.  Since
such Blaschke products may be infinite,
membership of $g$ in $A$ is not necessary for $T_g$ to be bounded on $H^{\infty}$.
It turns out to also be not sufficient, yet the interesting examples may shed some light on the problem, especially in the univalent case.  First we show that there exist univalent functions
$g\in T[H^\infty]\setminus A$.
 Key to our argument is the
 next theorem, concerning length distortion by a conformal map.  It is a version,
 suitable for our application, of
the important Gehring-Hayman Theorem.  We use $\ell(E)$ to denote the arc length of a rectifiable
curve $E$.

\begin{thm} \emph{\cite[p. 72]{Pom}}
\label{hypgeo}
Let $f: \D \to G$ be analytic and univalent.  Let $E \subseteq G$ be a rectifiable curve from
$f(0)$ to $f(z)$, where $z = re^{i\theta}$. There is an absolute constant $K$ such that
  $$\int_0^r |f'(te^{i\theta})| \, dt \leq K \ell(E).$$
\end{thm}

\begin{prop}
\label{gnotinA}
There exist univalent $g \notin A$ such that $T_g$ is bounded.
\end{prop}

\begin{proof}
Our example is a Riemann map to a comblike domain.
Denote the unit square
  $$R = \{x + iy : 0 < x < 1, 0 < y < 1\}$$
and the set of segments
  $$S = \{2^{-n} + iy : 1/2 < y < 1, n = 1,2,... \}.$$
Let $G$ be the region $R\setminus S$ (see Figure \ref{combfig}), and let $g$ be a Riemann map
from $\D$ to $G$ with $g(0)=1/2 +i/4$.  Then $g$ is not in $A$, since the boundary
$\partial G$ is not locally connected \cite[Theorem 2.1]{Pom}.

Note that any point $g(z) \in G$ can be connected to the point $g(0)$ by a horizontal and a vertical line segment, comprising a rectifiable curve of length less than 2.
Thus  from Theorem \ref{hypgeo} we have
$$
V(g,\theta)=\lim_{r\to 1}\int_0^r |g'(te^{i\theta})| \, dt\le 2K, \quad 0\le\theta\le 2\pi.
$$
Hence $g\in T[H^\infty]$ by Proposition \ref{ubrv}.
\end{proof}

\begin{figure}
\begin{minipage}{6cm}
\begin{tikzpicture}[scale=1.5]
\draw (0,0) rectangle (4,4); 
\draw (2,4) --(2,2) (1,4) --(1,2) (0.5,4) --(0.5,2); 
\draw [dashed] (0.25,4) --(0.25,2);
\draw  [fill]  (2,1) circle [radius=0.02]
node  [anchor=west] {$g(0)$};
\end{tikzpicture}
\caption{$g \notin A$, $g \in T[H^{\infty}]$}
\label{combfig}
\end{minipage}
\quad
\begin{minipage}{6cm}
\begin{tikzpicture}[scale=0.6]
\draw (0,0) circle (5cm);
\draw (-3.33333,0) arc (180:0:4.16667cm) ;
\draw (-3.33333,0) arc (180:360:2.91667cm) ;
\draw (-2,0) arc (180:0:2.25cm) ;
\draw (-2,0) arc (180:360:1.83333cm) ;
\draw (-1.42857,0) arc (180:0:1.54762cm) ;
\draw (-1.42857,0) arc (180:360:1.33929cm) ;
\draw (-1.11111,0) arc (180:0:1.18056cm) ;
\draw (-1.11111,0)[dashed] arc (180:360:1.05556cm) ;
\draw (-0.909091,0) [dashed] arc (180:0:0.954545cm) ;
\draw  [fill]  (5,0) circle [radius=0.05]
node  [anchor=west] {1};
\draw  [fill]  (2.5,0) circle [radius=0.05]
node  [anchor=west] {$\frac12$};
\draw [fill] (1.66667,0)  circle [radius=0.05]
node [anchor=west] {$\frac13$};
\draw  [fill]  (0,0) circle [radius=0.05]
node  [anchor=west] {0};
\end{tikzpicture}
\caption{$g \in A$, $g \notin T[H^{\infty}]$}
\label{deathspiralfig}
\end{minipage}
\end{figure}

Next,  we give a  proposition that provides an equivalent formulation of when a univalent function
$g$ is in $T[H^\infty]$; see also  \cite[ p. 2]{Sis}. This will then be used
to give an example that shows $A \nsubseteq T[H^{\infty}]$.

Let $\Omega$ be a simply connected domain and let $w_0\in\Omega$.   Define the
linear operator taking a function $f$ holomorphic on $\Omega$ to its indefinite integral by
\begin{equation}\label{IntOpJ}
J_{w_0}f(w)=\int _{w_0}^wf(t)dt, \quad w\in\Omega,
\end{equation}
where integration is over any smooth curve in $\Omega$ connecting $w_0$ to $w$. Then $J_{w_0}f$
is holomorphic and well defined since $\Omega$ is simply connected.  Let $H^\infty(\Omega)$
denote the usual space of bounded holomorphic functions on $\Omega$, with supremum norm.

\begin{prop}\label{univalent}
Let $\Omega$ be a simply connected proper subdomain of the plane and let $g$  be a conformal map from $\D$ onto $\Omega$.  Then $g \in T[H^{\infty}]$ if and only if $J_{g(0)}$ is a bounded operator
on $H^\infty(\Omega)$.
\end{prop}

\begin{proof}
Let $g$  be a conformal map from $\D$ onto $\Omega$.  A change of variable shows that, for
$f$ holomorphic on $\Omega$ and $F=J_{g(0)}f$,
\begin{align}\label{IntOp}
T_g(f \circ g)(z) = \int_{0}^{z} F'(g(t)) g'(t) \; dt
 = F(g(z)).
 \end{align}
Since composition with $g$ is an isometry from $H^\infty(\Omega)$ onto $H^\infty(\D)$,
the result follows.
\end{proof}


\begin{prop}\label{deathspiral}
There exists a univalent function $g\in A$ such that $T_g$ is not bounded.
\end{prop}

\begin{proof}
Let $\Omega = \D \backslash \overline{ \{\gamma(t): t \geq 1 \}}$, where $\displaystyle{\gamma(t) = e^{2 \pi i t} / t, \; t \geq 1}$ (see Figure \ref{deathspiralfig}).
Note that $\displaystyle{\gamma(n) = 1/n}$ ,  for  $n \in \mathbb N$.  Clearly $\Omega$ is a simply connected domain with locally connected boundary and a prime end at $0$.  Let $g$ be a Riemann map from $\D$ onto $\Omega$.  Then $g\in A$ since $\partial\Omega$ is locally connected
\cite[Theorem 2.1]{Pom}, but we will show that $g\notin T[H^\infty]$.

Let $\displaystyle{r_n = \frac{1}{2} \left( \frac{1}{n} + \frac{1}{n+ 1} \right) }$,  $n \in \mathbb N$, so $r_n\in\Omega$.  Since $\Omega$ is simply connected and $0 \notin \Omega$,
there is a branch $\ell(z)$ of $\log(z)$ on $\Omega$ such that
\begin{align*}
\ell( r_n)  = \log|r_n| + 2 \pi i (n +1).
\end{align*}
Since $\im \ell(z)>0$, $z\in\Omega$,  we can
define $H(z) =$ Log$ (\ell( z))$ on $\Omega$, where Log  is the principal branch of the logarithm.  Then
$$|H( r_n)|\ge|\re H( r_n)|  =  \log |\ell (r_n)| \sim \log n, $$
 so $H \notin H^{\infty}(\Omega)$.  Here $a_n \sim b_n$ means that $\displaystyle{\frac{a_n}{b_n} \rightarrow 1}$, as $n \rightarrow \infty$.

Next, observe that for $z \in \Omega$ such that $\gamma(n+1) \leq |z| < \gamma(n)$,  we have that  $\im\ell (z) \sim 2 \pi n$ and $|z| \sim 1/n$.  Hence
$H'(z) = (z \ell(z))^{-1} \in H^{\infty}(\Omega)$.  Since
  $$J_{g(0)} H' = H - H(g(0)),$$
$J_{g(0)}$ is not bounded on $H^{\infty}(\Omega)$, and $g \notin T[H^{\infty}]$ by Proposition \ref{univalent}.
\end{proof}

\begin{cor} \label{Tnotclosed}
$T[H^\infty]$ is a subalgebra of $H^\infty$, but is not closed.
\end{cor}
\begin{proof}  As noted in the Introduction, $S[X]$ is always an algebra.
Since $T[H^\infty]=S[H^\infty]$, $T[H^\infty]$ is a subalgebra of $H^\infty$.
Clearly every polynomial $p\in T[H^\infty]$.  If $T[H^\infty]$ were a closed
subspace of $H^\infty$, then it would contain $A$.
By Proposition \ref{deathspiral}, this is not the case, which finishes
the proof.
\end{proof}

\begin{rem}\label{deathspiral2}
Note that if $\Omega = \D \backslash \overline{ \{\gamma(t): t \geq 0 \}}$, where $\displaystyle{\gamma(t) = e^{2 \pi i t} / 2^t, \; t \geq 0}$, and $g$ is a conformal map from $\D$ onto $\Omega$, then $g \in T[H^{\infty}]$  by Proposition \ref{ubrv} and Theorem \ref{hypgeo}.
\end{rem}

Of course, when endowed with the supremum norm $A$ is a closed subspace of $H^\infty$, and it is natural to
consider how  $T[A]$ is related to $T[H^\infty]$.

\begin{prop} \label{TAsubTH}
$T[A] = T[H^{\infty}] \cap A$.
\end{prop}

\begin{proof}
Suppose $g \in T[A]$.  Applying $T_g$ to 1 shows $g \in A$.  Let $f \in H^{\infty}$ and let $f_r(z) = f(rz)$ for $0 < r < 1$.  We have
  $$f_r \in A \text{ and } \|f_r\|_{\infty} \leq \|f\|_{\infty}, \quad 0 < r < 1.$$
For $z \in D$, $\lim_{r \to 1} \int_0^z f_r(t)g'(t) \, dt = \int_0^z f(t)g'(t)\, dt$ by the Bounded Convergence Theorem.  Hence
\begin{align*}
  |T_gf(z)| = \left| \lim_{r \to 1} T_gf_r(z) \right|
  \leq \|T_g\| \|f\|_{\infty}.
\end{align*}
Therefore $T_g$ is bounded on $H^{\infty}$, which completes the proof that $T[A] \subseteq T[H^{\infty}] \cap A$.

To prove the reverse inclusion, suppose $g \in T[H^{\infty}] \cap A$ and let $f\in A$.
Let $0<r<1$, so $f_r$ extends to be analytic across the unit circle.
Then $g\in A$ implies that the functions $S_gf_r$ and $ M_gf_r$ are in $A$.
Hence $T_g f_r= M_gf_r-S_g f_r-g(0)f(0)\in A$.

Also,
  $$\lim_{r \to 1^-} \|T_gf - T_g f_r\|_{\infty}
  \leq \lim_{r \to 1^-} \|T_g\|_{H^\infty} \|f - f_r\|_{\infty} = 0.$$
Since $A$ is a closed subspace of $H^\infty$, it follows that $T_g f\in A$.
Hence $g \in T[A]$, completing the proof.
\end{proof}

\section{$T_o[H^{\infty}]$}

We now discuss compactness, beginning with a characterization of when one of the operators
we are studying is compact on $H^\infty$.  First, we introduce the notation
 $$
 \overline{B} = \{f \in H^{\infty}: \| f\|_{\infty} \leq 1\}
 $$
 for the closed unit ball of $H^{\infty}$.

\begin{prop} \label{cptcriterion}
Let $L$ be one of the operators $T_g$, $S_g$, or $ M_g$ acting on $H^\infty$.
If $L$ is bounded, then the following are equivalent:
\begin{enumerate}
  \item[(i)]
 $L$ is compact on $H^{\infty}$;

\item[(ii)] If $\{f_n\}\subseteq\overline{B}$ and $f_n(z) \to 0$ locally uniformly in $\D$, then $\|Lf_n\|_{\infty} \to 0$.
\end{enumerate}

\end{prop}

\begin{proof}  Assume that  $\{f_n\}\subseteq\overline{B}$, that $f_n(z) \to 0$
locally uniformly in $\D$, and that $L$ is compact.  To prove (ii) holds,
by a standard argument it suffices to
show that there is a subsequence $\{f_{n_k}\}$ such that $\|L f_{n_k}\|_\infty\to 0$.
It is easy to see from the definitions of the operators
that if $L$ is one of $T_g$, $S_g$, or $ M_g$, then $Lf_n(z)\to0$
locally uniformly in $\D$.
Since $L$ is compact, there is a subsequence $\{f_{n_k}\}$
and $h\in H^\infty$ such that $\|L f_{n_k}- h\|_\infty\to 0$.  Since $L f_{n_k}(z)\to0$
locally uniformly in $\D$, $h=0$ and hence $\|L f_{n_k}\|_\infty\to 0$.  This completes the
proof that (i) implies (ii).

Next, assume that (ii) holds and let $\{h_n\} \subseteq\overline{B}$.  Then
 $\{h_n\}$ is a normal family, and hence there exists $h \in H^{\infty}$ and a subsequence $\{h_{n_k}\}$
such that $h_{n_k} \to h$ locally uniformly in $\D$.
Let $\{f_k\} = \{h_{n_k} - h\}$.
By (ii), $\|Lf_k\|_{\infty} = \|Lh_{n_k} - Lh\|_{\infty} \to 0$.  Since by assumption
$L$ is bounded, $Lh\in H^\infty$. Hence $Lh_{n_k}$ converges in
$H^{\infty}$ to $Lh$, which completes the proof that $L$ is compact.
\end{proof}

 \begin{rem}
 {\rm We remark that}
\begin{enumerate}
\item[(a)]  The sufficiency of condition  (ii) in Proposition \ref{cptcriterion} for $L$ to be compact on $H^\infty$ is valid for
any bounded linear operator.
 \item[(b)] The necessity of condition (ii) in Proposition \ref{cptcriterion} for $L$ to be compact on $H^\infty$ is not valid for general bounded linear operators.
 \end{enumerate}
 \end{rem}

  Indeed, the proof of sufficiency given above is valid for any operator.
  For an example showing that necessity fails in general, start with the evaluation
  functional $\Lambda$ defined
 on the disk algebra $A$ by $\Lambda f=f(1)$.  Then $\Lambda$ is a norm 1 linear functional on $A$, and by
 the Hahn-Banach Theorem can be extended to a linear functional $\hat \Lambda$ on $H^\infty$ with
 $\|\hat \Lambda\|=1$.  Let $L$ be the operator on $H^\infty$ that takes $f\in H^\infty$ to the constant
 function with constant value $\hat \Lambda f$.  Then $L$ is a rank-one operator, and hence compact.
 But $L(z^n)=  \Lambda (z^n)=1$ for all positive integers $n$, and hence condition (ii) fails.

\medskip

As with boundedness, if two of the operators $M_g$, $S_g$ and $T_g$ are compact, then so is the third.  Since $M_o[H^{\infty}] = \{0\}$, it is not surprising that the same is true for $S_g$.

\begin{prop} \label{S_o}
$S_o[H^\infty]=S_o[A]=\{0\}$.
\end{prop}

\begin{proof}
If $S_g$ is compact on $H^{\infty}$,  then $S_g$ is bounded and $g \in H^{\infty}$ by Proposition \ref{S_gbdd}.  If $g\neq0$, then without loss of generality we may assume $\|g\|_{\infty} = 1$.  Note that $g$ is not constant, for otherwise $S_g$ would be a constant multiple of a rank-1 perturbation of the identity operator, which is not compact.  Thus, the functions $g^n$, $n = 1,2,...$ converge locally uniformly to 0 in $\D$.  We have
\begin{align*}
  S_g g^n(z) = \int_0^z g(t) (g^n)'(t)\, dt
  =\frac{n}{n+1} g^{n+1}(z) - \frac{n}{n+1}g^{n+1}(0).
\end{align*}
Since $g^{n+1}(0) \to 0$, as $n \rightarrow \infty$,  but $\|g^{n+1}\|_{\infty} = 1$ for all $n$, this violates the condition for compactness
Proposition \ref{cptcriterion} (ii).  Hence $g=0$, showing that $S_o[H^\infty]=\{0\}$.

The same proof shows that $S_o[A]=\{0\}$, and so will be omitted.
\end{proof}

Although $S_o[H^\infty]=M_o[H^\infty]=\{0\}$,
there are non constant functions $g$ for which $T_g$ is compact.  We now introduce notation
for spaces of functions which we will show have this property.  Let
$$
H^1_1=\{g\in H(\D)\,:\, g'\in H^1\}
$$
so, for example, a conformal map from $\D$ to a domain with rectifiable boundary belongs to $H^1_1$;
see \cite[Theorem 3.12]{Dur}.  Let
$$
\ell^1(\D)=\left\{g(z)=\sum_{n=0}^\infty a_n z^n\,:\,\sum_{n=0}^\infty |a_n|<\infty \right\}
$$
denote those functions holomorphic on $\D$ with absolutely convergent Fourier series.
Next, we say that the derivative $g'$ of a function $g\in H(\D)$ is uniformly
integrable on radii if: Given $\eps > 0$, there exists $r < 1$
independent of $\theta$ such that
\begin{equation} \label{UI}
\int_r^1 |g'(te^{i\theta})|\, dt \leq  \eps , \qquad 0 \leq \theta < 2\pi.
\end{equation}
We use this condition to define the final space of functions we consider:
$$
\mathcal U=\{g\in H(\D)\,:\, g'\text{ is uniformly integrable on radii}\}.
$$

\begin{prop} \label{cptsuff}
$H^1_1\subsetneqq \ell^1(\D)\subsetneqq \mathcal U\subseteq T_o[H^\infty]$.

\end{prop}

Before giving the proof we note that the Fej\'er-Riesz inequality \cite[Theorem 3.13]{Dur}
tells us that
$$
\int_0^1 |g'(te^{i\theta})|\, dt \leq \pi\|g'\|_{H^1}.
$$
Thus the inclusion $H^1_1\subseteq \mathcal U$ can be viewed as a uniform integrability version
of this classical inequality.

\begin{proof}

Let  $g(z) = \sum_0^{\infty} a_k z^k$, so $g'(z) = \sum_0^{\infty} (k+1)a_{k+1} z^k$.  By Hardy's inequality \cite[p. 48]{Dur}, $g' \in H^1$ implies
\begin{align*}
  \sum_{k =1}^{\infty} |a_k|
  = \sum_{k=0}^{\infty} \frac{(k+1) |a_{k+1}|}{k+1}
  \leq \pi \| g' \|_{H^1}.
\end{align*}
Hence $H^1_1\subseteq \ell^1(\D)$.  To see that the inclusion is proper, consider the
lacunary series $g(z) = \sum_0^\infty 2^{-k} z^{2^k}$.  Then $g\in\ell^1(\D)$, while
the Riemann-Lesbesgue Lemma shows that $g'(z) = \sum_0^\infty z^{2^k-1}$ is not in $H^1$.

Next, let $g(z) = \sum_{0}^{\infty} a_k z^k\in \ell^1(\D)$, and
let $\eps > 0$.  Set $b_k = (k+1)a_{k+1}$, so $g'(z) = \sum_{0}^{\infty} b_kz^k$. Since
$g\in \ell^1(\D)$,
there exists a positive integer $N$ such that
  $$\sum_{k=N}^{\infty}\frac{|b_k|}{k+1} < \frac{\eps}{2}.$$
Let
  $$p_N(z) = \sum_{k=0}^{N-1} b_kz^k, $$
and choose $r < 1$ such that
  $$\int_r^1 |p_N(te^{i\theta})|\, dt < \eps/2,\qquad 0 \leq \theta < 2\pi.$$
Then for $0 \leq \theta < 2\pi$ we have
\begin{align*}
\int_r^1| g'(te^{i\theta})| dt
&\leq \int_r^1 \left| p_N(te^{i\theta}) \right| \, dt + \int_r^1  \sum_{k=N}^{\infty} |b_k| t^k \, dt \\
&< \frac{\eps}{2} + \sum_{k=N}^{\infty} \frac{1 - r^{k+1}}{k+1} |b_k| < \eps.
\end{align*}
Hence (\ref{UI}) is satisfied, which completes the proof that $\ell^1(\D)\subseteqq\mathcal U$.
To see that this inclusion is proper, it is known that there exists $f\in H(\D)$ such that
$|f'(z)|\le C(1-|z|)^{-1/2}$, $|z|<1$, and hence $f\in \mathcal U$, but $f\notin\ell^1(\D)$; see
\cite[Chapt. 5 ex. 7]{Dur}.

To prove the final inclusion, let $g\in \mathcal U$.
If $g' = 0$ then $g\in T_o[H^\infty]$ trivially, so assume $g' \neq 0$.  Let $\eps > 0$, and choose $r, 1- \eps < r < 1,$ such that (\ref{UI}) holds. If $\{f_n\} \subseteq\overline{B}$ and $f_n \to 0$ uniformly on compact subsets of $\D$, then there exists $N$ such that
  $$|f_n(z)| < \frac{\eps}{\sup \{|g'(z)|: |z| \leq r\}}, \quad |z| \leq r, \quad n > N.$$
Then for $n > N$,
\begin{align*}
  \|T_gf_n\|_{\infty} &=\sup_{\theta} \left| \int_0^1 f_n(te^{i\theta})g'(te^{i\theta})\, dt \right|\\
  &\leq \sup_{\theta} \left( \int_0^r | f_n(te^{i\theta})g'(te^{i\theta})|\, dt + \int_r^1 |f_n(te^{i\theta})g'(te^{i\theta})| \, dt\right)\\
  &\leq  \eps r + \eps.
\end{align*}
Thus, $\|T_gf_n\|_{\infty} \to 0$, and $T_g$ is compact by Proposition \ref{cptcriterion}.
\end{proof}


\begin{thm} \label{TosubsetA}
$T_o[H^\infty]\subsetneqq A$.
\end{thm}

\begin{proof}
For $f \in H^{\infty}$, $\delta \in \R$, and $z \in \D$, define
  $$f^{\delta}(z) = f(ze^{i \delta}).$$
Suppose $T_g$ is compact on $H^{\infty}$.  Since $f^{\delta}(z) \to f(z)$ uniformly on compact subsets of $\D$, Proposition \ref{cptcriterion} implies
  $$\|T_g(f - f^{\delta})\|_{\infty} \to 0 \text{ as } \delta \to 0.$$
Also,
\begin{align*}
T_{g^{\delta}}(f - f^{\delta})(z)
&= \int_0^z g'(we^{i\delta})(e^{i\delta})(f(w) - f(we^{i\delta})) \, dw\\
&= \int_0^{ze^{i\delta}} g'(u)(f(ue^{-i\delta}) - f(u)) \, du \qquad{ (u = we^{i\delta}) }\\
&= T_g(f^{-\delta} - f)(ze^{i\delta}),
\end{align*}
so
  $$\| T_{g^{\delta}}(f - f^{\delta}) \|_{\infty} \to 0 \text{ as } \delta \to 0.$$
Thus, by linearity of $T_g$ in the symbol $g$,
  $$\|T_{g - g^{\delta}} (f - f^{\delta})\|_{\infty} = \|T_g(f - f^{\delta}) - T_{g^{\delta}}(f - f^{\delta}) \|_{\infty} \to 0 \text{ as } \delta \to 0.$$
Setting $f = g$ and $h = f - f^{\delta} = g - g^{\delta}$, since
  $$ T_h h(z) =  \int_0^z h(w)h'(w)\, dw = (h(z))^2/2,$$
we obtain
  \begin{equation} \label{gdelta}
  \|g - g_{\delta}\|_{\infty} \to 0 \text{ as } \delta \to 0.
  \end{equation}

An argument involving the modulus of continuity of $g$ will complete the proof; see, eg., \cite{RST}.
For $f \in A, \delta > 0$, let
  $$\omega(\delta,f) = \sup \{ | f(z_1)-f(z_2)| : |z_1 - z_2| < \delta, \quad z_1, z_2 \in \overline{\D} \},$$
and
  $$\Tilde{\omega}(\delta,f) = \sup \{  | f(z_1)-f(z_2)| : |z_1 - z_2| < \delta, \quad z_1, z_2 \in \partial \D \}.$$

For $0 < r < 1$, denote the functions
  $$g_r(z) = g(rz).$$
Let $\eps > 0.$  (\ref{gdelta}) implies the existence of $\delta > 0$ independent of $r$ such that
  $$\Tilde{\omega} (\delta,g_r) < \eps.$$
By \cite[Theorem 1.1]{RST},
  $$\omega(\delta, g_r) \leq 3 \Tilde{\omega}(\delta, g_r) < 3 \epsilon.$$
Since this estimate holds independent of $r \in (0,1)$,  $g$ is uniformly continuous in $\D$, i.e., $g \in A$. This completes the proof that $T_o[H^\infty]\subset A$.
That the containment is proper follows from Proposition \ref{deathspiral}.
\end{proof}

We saw in Proposition \ref{gnotinA} that $T[H^{\infty}]\setminus A$ is non-empty.
This gives the next corollary.

\begin{cor} \label{BnotinTo}
$T_o[H^{\infty}] \subsetneqq T[H^{\infty}]$.
\end{cor}

Next we turn to the relationship between $T_o[H^{\infty}]$ and $T_o[A]$.
\begin{prop} \label{To[A]}
$T_o[H^{\infty}] \subseteq T_o[A].$
\end{prop}
\begin{proof}
Suppose $g \in T_o[H^{\infty}]$.
For any polynomial $p$,
  $$S_gp (z)= \int_0^zg(t)p'(t) \, dt \in A,$$
since $p'g\in H^\infty$.  Also,
$g \in A$ by Theorem \ref{TosubsetA}, and so $M_g p\in A$.  Hence,    $T_gp = M_gp - S_gp-g(0)p(0)\in A.$  Since the polynomials are dense in the closed subspace
$A$ of $H^\infty$, and $T_g$ is bounded on $H^\infty$, it follows that $T_g$ maps $A$ to $A$
and is  compact on $A$.
\end{proof}

\begin{cor}\label{T_otoA}
If $g \in T_o[H^{\infty}]$, then  $T_g : H^{\infty} \to A$.
\end{cor}

\begin{proof}
Let $g \in T_o[H^{\infty}]$.  
For  $f \in H^{\infty}$ define
$$\displaystyle{f_n(z) = f \left( \left (1 - \frac{1}{n} \right )z \right)}, \;\;n \geq 1.$$
Clearly $f_n \in A$,  $||f_n||_A \leq ||f||_{H^{\infty}}$  for  $n \geq 1$,  and $f_n \rightarrow f$
locally uniformly on $\D$ as $n \rightarrow \infty$.  Since $T_g$ is compact on $A$ (by Proposition \ref{To[A]}), there exists a subsequence $\{f_{n_k}\}$ of $\{f_n\}$  such that $T_g(f_{n_k}) \rightarrow h$ in $A$.  It follows that $T_g(f_{n_k}) \rightarrow h$  pointwise.

For fixed $z \in \D$ we have
$$\lim_{n \rightarrow \infty} T_g f_n (z)  = \lim_{n \rightarrow \infty} \int_0^z g'(t)f_n(t) \; dt = \int_0^z g'(t)f(t) \; dt = T_g f(z).$$
Therefore $T_g f (z) = h(z)$, and hence $T_g f \in A$.
\end{proof}

\section{Problems}

In this section we collect some problems left unresolved in our work:

\begin{prob} \label{BddQ}
Give a function theoretic characterization of $T[H^{\infty}]$.
\end{prob}

From Proposition \ref{ubrv}, we know $\rm{BRV}\subseteq T[H^{\infty}]$.  This leads to the
natural question:
\begin{quote}
{\em Is $\rm{BRV}= T[H^{\infty}]$?}
\end{quote}

Next, we present a problem in geometric function theory that is a version of
 Problem \ref{BddQ} for univalent functions.  For a simply connected domain
 $\Omega$ and $w_0\in\Omega$,
let $J_{w_0}$ be the
integration operator  defined in (\ref{IntOpJ}).  Clearly, if $J_{w_0}$ is bounded on
$H^\infty(\Omega)$, then  $J_{w_1}$ is also bounded for all $w_1\in\Omega$.

\begin{prob}
Give a geometric characterization of simply connected domains $\Omega$ such that $J_{w_0}$
is bounded (compact) on $H^\infty(\Omega)$.
\end{prob}

The question given after Problem \ref{BddQ} can be formuated in this
setting.
Define the {\em arc-length distance} between the points
$z$ and $w$ in a domain $\Omega$ in the plane to be the infimum of the arc-lengths of rectifiable curves in $\Omega$ that connect $z$ to $w$.  The {\em arc-length diameter} of $\Omega$ is the supremum
of the arc-length distances between points in $\Omega$.  Using Theorem \ref{hypgeo} we see that
the image of $\D$ under a univalent map $g$ has finite arc-length diameter if and only
if $g\in \rm{BRV}$, and recall from Proposition \ref{univalent} that $g\in T[H^\infty]$ if and
only if $J_{g(0)}$ is bounded on $H^\infty(\Omega)$.  Thus we are led to the question:

\begin{quote}
{\em
Is the operator $J_{w_0}$ bounded on $H^\infty(\Omega)$ if and only if
$\Omega$ has finite arc-length diameter?}
\end{quote}

The next problem concerns Blaschke products.
We saw that uniformly Frostman Blaschke products are in BRV.  It is natural
to ask if there are other infinite Blaschke products, or families of Blaschke products,
that belong to this class.

\begin{prob}
Characterize, in terms of their zero sequences,  Blaschke  products in BRV.
\end{prob}

A possible candidate for a Blaschke product in BRV is the Blaschke product with zeros $\{1 - 1/k^2 \}$.
It is not hard to see that the image of every radius is rectifiable, but it is not clear that there is a
uniform bound for their arc-lengths.

We end with the problem of characterizing when $T_g$ is compact.

\begin{prob} \label{cptprob}
Give function theoretic characterizations of $T_o[H^{\infty}]$ and $T_o[A]$.
\end{prob}

We know  that $\mathcal U \subseteq T_o[H^{\infty}]\subseteq T_o[A] \subseteq T[A]\subsetneqq A$,
but we do not know if the first three inclusions are proper.

 \end{document}